\documentclass[11pt]{article}
\usepackage[T1]{fontenc}
\usepackage[utf8]{inputenc}
\usepackage{amsmath,amssymb,amsthm,microtype}
\usepackage[hidelinks]{hyperref}
\hypersetup{
  pdftitle={Hamel-based TVG Stronger than the Usual Topology on R},
  pdfauthor={Irina Yaroshevskaya},
  pdfsubject={General topology and topological vector groups},
  pdfkeywords={topological vector group, Hamel basis, additive function, initial topology}
}
\newtheorem{theorem}{Theorem}
\newtheorem{lemma}[theorem]{Lemma}
\newtheorem{proposition}[theorem]{Proposition}
\newtheorem{corollary}[theorem]{Corollary}
\title{Hamel-based TVG Stronger than the Usual Topology on $\mathbb R$}
\author{Irina Yaroshevskaya}
\date{}
\begin{document}
\maketitle
\begin{abstract}
For each nonzero rational-linear functional $f:\mathbb R\to\mathbb Q$, we construct a Hausdorff nondiscrete topological vector group topology $\tau_f$ strictly finer than the usual topology on $\mathbb R$. We prove that $\tau_f=\tau_g$ if and only if $g(x)=f(bx)$ for some nonzero real $b$, and that distinct members of this family are incomparable. There are exactly $2^{\mathfrak{c}}$ distinct topologies on the fixed real line, where $\mathfrak{c}=|\mathbb R|$. Every member is totally disconnected and fails to be first countable. The supremum of the family is still nondiscrete and strictly finer than each individual member. We also prove the existence of maximal nondiscrete TVG extensions.
\end{abstract}
\noindent\textbf{2020 Mathematics Subject Classification.} Primary 54H11; Secondary 22A05, 54A10.
\par\noindent\textbf{Key words and phrases.} topological vector group, Hamel basis, additive function, initial topology, incomparable topologies.

\section{Introduction and definitions}
We construct nondiscrete Hausdorff topological vector group topologies strictly finer than the usual topology on the real line. A rational coordinate with respect to a Hamel basis supplies an additive function; imposing continuity of all its real rescalings makes every fixed dilation continuous. We then classify equality and inclusion within this family and count the resulting topologies.

In \cite{Yaroshevskaya1975} we introduced a Hausdorff quasi-vector topology $t$ on $\mathbb R$ distinct from the usual topology. In \cite{Yaroshevskaya1981}, received in 1979 and published in 1981, we proved that $t$ is minimal among Hausdorff TVG topologies: no strictly coarser Hausdorff TVG topology exists. More recently, \cite{Yaroshevskaya2026} constructs topologies $\sigma$ satisfying $t\subsetneq\sigma\subsetneq\tau_u$. The present article studies TVG topologies above $\tau_u$, and proves an equality criterion and a multiplicity theorem for the Hamel coefficient construction.

Write $\mathbb Z$ for the integers, $\mathbb Q=\{\frac{p}{q}:p,q\in\mathbb Z,\ q\ne0\}$ for the rational numbers, and $\mathbb R$ for the real numbers. Denote the usual and discrete topologies on $\mathbb R$ by $\tau_u$ and $\tau_d$, respectively. A topology is finer, or stronger, when it contains the other topology as a collection of open sets.

A topological vector group (TVG) topology on $\mathbb R$ is a group topology for addition such that, for each fixed $a\in\mathbb R$, the map $x\mapsto ax$ is continuous. Joint continuity of the map $(a,x)\mapsto ax$, with the scalar variable carrying its usual topology, is not assumed. This is the convention used in \cite{Dergachev1978,DominguezTarieladze2008}.

A Hamel basis of $\mathbb R$ over $\mathbb Q$ is a subset $H\subset\mathbb R$ such that every real number has a unique finite rational expansion

\[
x=\sum_{h\in F}q_hh,\qquad F\subset H\text{ finite},\quad q_h\in\mathbb Q.
\]
Equivalently, $H$ spans $\mathbb R$ over $\mathbb Q$ and is linearly independent over $\mathbb Q$. Choose $h_0\in H$ and let $f(x)$ be the coefficient of $h_0$ in this expansion, with coefficient zero when $h_0$ is absent. Then

\[
f(x+y)=f(x)+f(y),\qquad f(qx)=qf(x)\quad(q\in\mathbb Q),
\]
\[
f(h_0)=1,\qquad f(h)=0\quad(h\in H,\ h\ne h_0).
\]
The existence of Hamel bases and their extension from linearly independent sets uses the axiom of choice. Hamel's original work \cite{Hamel1905} establishes the classical connection with discontinuous solutions of the additive functional equation.

Put $D=\operatorname{Hom}_{\mathbb Q}(\mathbb R,\mathbb Q)$, the set of all rational-linear maps from $\mathbb R$ to $\mathbb Q$. Every nonzero $f\in D$ is a Hamel coefficient functional: choose $h_0$ with $f(h_0)=1$, choose a basis of $\ker f$, and adjoin $h_0$. We may therefore work with arbitrary nonzero members of $D$.

\section{The construction}\label{sec:construction}
For $f\in D\setminus\{0\}$ and $a\in\mathbb R$, define $f_a(x)=f(ax)$. Let $\tau_f$ be the coarsest topology making the identity map into $\mathbb R_u$ and every map $f_a$ into $\mathbb R_u$ continuous. Here $\mathbb R_u$ means the usual real line, and rational values are viewed as real values.

Basic neighborhoods of zero are

\[
U_f(\delta;A,\varepsilon)=\{x:|x|<\delta,\ |f(ax)|<\varepsilon\text{ for every }a\in A\},
\]
where $\delta,\varepsilon>0$ and $A\subset\mathbb R$ is finite, possibly empty. Using one common $\varepsilon$ gives a base because any finite list of positive bounds can be replaced by its minimum. Neighborhoods at $y$ are translates $y+U_f(\delta;A,\varepsilon)$.

\begin{theorem}[Construction]\label{thm:construction}
For every nonzero $f\in D$, the topology $\tau_f$ is a Hausdorff nondiscrete TVG topology, and
\[
\tau_u\subsetneq\tau_f\subsetneq\tau_d.
\]
\end{theorem}
\begin{proof}
All the defining maps are additive, so their initial topology is a group topology. More explicitly, halving both positive bounds in a basic neighborhood $U$ gives a neighborhood $V$ with $V+V\subset U$, and $U=-U$. If $y\in U$, the finitely many strict inequalities have positive margins; choose smaller positive bounds so that the corresponding translated neighborhood at $y$ is contained in $U$. This also verifies that the displayed sets are open.

For a fixed $b\ne0$, multiplication by $b$ has inverse image

\[
\{x:|bx|<\delta,\ |f(abx)|<\varepsilon\ (a\in A)\}=U_f(\frac{\delta}{|b|};bA,\varepsilon).
\]
Thus every fixed nonzero scalar acts by a homeomorphism. The zero scalar gives a constant map. Taking $A$ empty shows that usual intervals are open, so $\tau_u\subseteq\tau_f$ and $\tau_f$ is Hausdorff.

A usual-continuous additive real function has the form $x\mapsto cx$: its values on the rationals determine its values everywhere by density. If such a function has its entire image in $\mathbb Q$, it is zero, since $c\mathbb R=\mathbb R$ for $c\ne0$. Consequently our nonzero rational-valued $f$ is not usual-continuous. It is $\tau_f$-continuous by construction, so $\tau_u\subsetneq\tau_f$.

Finally, fix a basic neighborhood. Rational linearity gives

\[
f(\frac{a}{n})=\frac{f(a)}{n}\longrightarrow0\qquad(n\longrightarrow\infty).
\]
Together with $\frac{1}{n}\to0$ usually, these finitely many limits show that $\frac{1}{n}\to0$ in $\tau_f$. Since every term is nonzero, $\tau_f$ is nondiscrete.

\end{proof}
\begin{proposition}[Separation and scalar continuity]\label{prop:separation}
Each $(\mathbb R,\tau_f)$ is totally disconnected. For every $v\ne0$, the orbit map $a\mapsto av$ from $\mathbb R_u$ to $(\mathbb R,\tau_f)$ is discontinuous at zero. In particular, scalar multiplication with the usual scalar topology is not jointly continuous.
\end{proposition}
\begin{proof}
If $x\ne y$, choose $t$ with $f(t)\ne0$ and set $a=\frac{t}{x-y}$. Then $f_a(x)\ne f_a(y)$. A connected subset has connected image under each $f_a$, but a connected subset of $\mathbb Q$ in its usual topology is a singleton: an irrational cut separates any two distinct rational numbers. Thus no connected subset contains two points.

For fixed $v\ne0$, continuity of the orbit map $a\mapsto av$ from $\mathbb R_u$ into $(\mathbb R,\tau_f)$, followed by the $\tau_f$-homeomorphism $x\mapsto \frac{x}{v}$, would make the identity from $\mathbb R_u$ to $(\mathbb R,\tau_f)$ continuous. This contradicts strictness in Theorem~\ref{thm:construction}. Failure of continuity at zero follows as well, since these orbit maps are additive.

\end{proof}
\section{Continuous additive functions}
We first give the finite-dimensional fact needed to compare the topologies. In this section the target of every continuous additive function carries the usual real topology.

\begin{lemma}[Finite-coordinate representation]\label{lem:finite-coordinates}
Let $S$ be a family of additive maps $\mathbb R\to\mathbb R$, and let $\tau_S$ be their initial topology. An additive map $g:\mathbb R\to\mathbb R$ is $\tau_S$-continuous if and only if it belongs to the real linear span of finitely many members of $S$.
\end{lemma}
\begin{proof}
A finite real linear combination of continuous maps is continuous. Conversely, the zero function needs no representation, so assume that $g$ is nonzero. Continuity at zero gives $s_1,\ldots,s_m\in S$ and $\eta>0$ such that

\[
\max_{1\le j\le m}|s_j(x)|<\eta\quad\Longrightarrow\quad |g(x)|<1.
\]
If $p(x)=\max_{1\le j\le m}|s_j(x)|=0$, apply the implication to every positive integer multiple of $x$ to get $g(x)=0$. If $p(x)>0$, choose a positive rational $q$ with

\[
\frac{\eta}{2p(x)}<q<\frac{\eta}{p(x)}.
\]
Additive maps on $\mathbb R$ are rational-linear. Applying the implication to $qx$ yields

\[
|g(x)|\le\frac{2}{\eta}p(x).
\]
Define $T(x)=(s_1(x),\ldots,s_m(x))$ and define $\ell(T(x))=g(x)$. The zero-kernel observation makes $\ell$ well defined; the displayed estimate makes it a bounded rational-linear functional on $L=T(\mathbb R)\subset\mathbb R^m$.

The closure $E$ of $L$ is a real vector subspace: it is a closed additive subgroup closed under rational multiplication, and limits of rational scalar multiples give real scalar multiplication. The functional $\ell$ extends uniquely and continuously to $E$ by the bound. Its extension is additive and continuous, hence real-linear. Extend it linearly to the finite-dimensional space $\mathbb R^m$. For suitable real coefficients $c_j$ this gives $g=\sum_{j=1}^m c_js_j$.

\end{proof}
\begin{theorem}[Rational-valued continuous functions]\label{thm:rational-dual}
For $f\in D\setminus\{0\}$, the rational-valued additive maps continuous on $(\mathbb R,\tau_f)$ are exactly
\[
\{f_b:b\in\mathbb R\}.
\]
Moreover, the parameter $b$ is unique.

\end{theorem}
\begin{proof}
Each $f_b$ is continuous by definition. Conversely, let $g\in D$ be continuous. Lemma~\ref{lem:finite-coordinates}, applied to the identity and the defining maps, gives

\[
g(x)=c_0x+\sum_{j=1}^m c_jf(a_jx),\qquad c_j\in\mathbb R.
\]
Let $K=\bigcap_{j=1}^m\ker f_{a_j}$. It has finite codimension over $\mathbb Q$ in $\mathbb R$, so its dimension over $\mathbb Q$ is infinite. For $x\in K$ we have $g(x)=c_0x\in\mathbb Q$. If $c_0\ne0$, this would imply $K\subseteq c_0^{-1}\mathbb Q$, a space of rational dimension one. Therefore $c_0=0$.

Choose a rational basis $h_1,\ldots,h_k$ for the rational span of the finitely many maps $f_{a_j}$ and rewrite $g=\sum_{i=1}^k d_ih_i$ with real $d_i$. If this span is zero, then $g=0=f_0$. Otherwise the rational-linear map

\[
x\longmapsto(h_1(x),\ldots,h_k(x))\in\mathbb Q^k
\]
is onto. Indeed, a proper image subspace would be annihilated by a nonzero rational linear functional on $\mathbb Q^k$, contradicting independence of the $h_i$. Choose $x_i$ mapping to the $i$th coordinate vector. Then $d_i=g(x_i)\in\mathbb Q$.

Consequently $g$ is a rational linear combination of the $f_{a_j}$. Rational linearity and additivity combine this into one rescaling:

\[
\sum_{j=1}^m q_jf(a_jx)=f\left(\left(\sum_{j=1}^m q_ja_j\right)x\right).
\]
Thus $g=f_b$ for some $b$. If $f_b=f_c$ and $b\ne c$, then $f((b-c)x)=0$ for every $x$. Multiplication by $b-c$ is onto $\mathbb R$, forcing $f=0$, a contradiction. Hence $b=c$.

\end{proof}
\section{Equality and inclusion}
\begin{theorem}[Exact comparison]\label{thm:comparison}
For nonzero $f,g\in D$, the following conditions are equivalent: (i) $\tau_g\subseteq\tau_f$; (ii) $g$ is $\tau_f$-continuous; (iii) $g=f_b$ for a nonzero real number $b$; (iv) $\tau_g=\tau_f$.
\end{theorem}
\begin{proof}
Condition (i) implies (ii), since $g$ is continuous for $\tau_g$. Theorem~\ref{thm:rational-dual} gives (ii) implies (iii); the parameter cannot be zero because $g\ne0$. If $g=f_b$ with $b\ne0$, then

\[
g_a=f_{ba}\qquad(a\in\mathbb R).
\]
As $a$ ranges over $\mathbb R$, so does $ba$. The two defining families of maps, each including the identity, are therefore the same. This proves (iv), which implies (i).

\end{proof}
\begin{corollary}[Incomparability]\label{cor:incomparability}
Two distinct topologies in the family $\{\tau_f:0\ne f\in D\}$ are incomparable by inclusion.
\end{corollary}
\begin{proof}
Either inclusion would imply equality by Theorem~\ref{thm:comparison}.

\end{proof}
In terms of the chosen Hamel bases, equality holds precisely when their selected coefficient maps satisfy $g(x)=f(bx)$ for some $b\ne0$. This criterion concerns the coefficient maps, rather than equality of the bases themselves. For example, replacing an unselected basis element $h_1$ by $2h_1$ leaves the selected coefficient map unchanged. Also, replacing the entire basis $H$ by $b^{-1}H$ and selecting $b^{-1}h_0$ produces the coefficient map $x\mapsto f(bx)$ and hence the same topology.

\section{Number of topologies and countability}
Write $\mathfrak{c}=|\mathbb R|$ for the cardinality of the continuum. Distinct topologies below mean distinct collections of open subsets of the same underlying real line; no classification up to arbitrary homeomorphism is asserted.

\begin{theorem}[Multiplicity]\label{thm:multiplicity}
The Hamel coefficient family contains exactly $2^{\mathfrak{c}}$ distinct topologies. They are pairwise incomparable when distinct.
\end{theorem}
\begin{proof}
The rational dimension of $\mathbb R$ is $\mathfrak{c}$. Indeed, a vector space with an infinite basis over the countable field $\mathbb Q$ has cardinality equal to the maximum of the basis cardinality and the countable cardinal; the uncountability of $\mathbb R$ therefore forces its basis cardinality to be $\mathfrak{c}$.

A map in $D$ is determined by arbitrary rational values on a Hamel basis. Hence

\[
|D|=|\mathbb Q|^{\mathfrak{c}}=2^{\mathfrak{c}}.
\]
For completeness, the lower bound follows by using only the values zero and one. The upper bound follows from $|\mathbb Q|\le2^{\aleph_0}$ and the identity $\aleph_0\mathfrak{c}=\mathfrak{c}$.

By Theorem~\ref{thm:comparison}, the nonzero functions generating a fixed topology are exactly the orbit $\{f_b:b\ne0\}$, which has cardinality $\mathfrak{c}$ by uniqueness in Theorem~\ref{thm:rational-dual}. Partitioning a set of cardinality $2^{\mathfrak{c}}$ into classes of cardinality $\mathfrak{c}$ gives exactly $2^{\mathfrak{c}}$ classes. Explicitly, if the number of classes is $\kappa$, then

\[
2^{\mathfrak{c}}=\kappa\mathfrak{c}=\max\{\kappa,\mathfrak{c}\},
\]
so Cantor's strict inequality $\mathfrak{c}<2^{\mathfrak{c}}$ forces $\kappa=2^{\mathfrak{c}}$. Incomparability follows from Corollary~\ref{cor:incomparability}.

\end{proof}
\begin{proposition}[Failure of first countability]\label{prop:first-countability}
No topology $\tau_f$ is first countable, and consequently none is metrizable.
\end{proposition}
\begin{proof}
Suppose a countable neighborhood base exists. For each member choose a basic neighborhood contained in it. Let $A_0$ be the countable union of the finite scalar sets used in these choices. The topology generated by the identity and $\{f_a:a\in A_0\}$ is then $\tau_f$: it is coarser by definition, while the chosen neighborhoods ensure continuity of the identity in the reverse direction at zero and hence everywhere.

Apply Lemma~\ref{lem:finite-coordinates} to this countable generating family. The proof of Theorem~\ref{thm:rational-dual} shows that every continuous rational-valued additive function lies in the rational span of $\{f_a:a\in A_0\}$. That span is countable. But all $f_b$, for $b\in\mathbb R$, are continuous and are pairwise distinct by Theorem~\ref{thm:rational-dual}. This is a contradiction. Every metrizable topology is first countable.

\end{proof}
\section{A common nondiscrete refinement}
Let $\tau_D$ be the topology generated by the identity into $\mathbb R_u$ and all members of $D$ into $\mathbb R_u$.

\begin{theorem}[Greatest member of the generated family]\label{thm:common-refinement}
The topology $\tau_D$ is a Hausdorff nondiscrete TVG topology and
\[
\tau_D=\bigvee_{0\ne f\in D}\tau_f.
\]
It is strictly finer than each individual $\tau_f$. Thus no individual $\tau_f$ is maximal among nondiscrete TVG topologies.

\end{theorem}
\begin{proof}
Finite intersections of inverse images under additive maps give a group-topology base. For every $g\in D$ and real $a$, the function $x\mapsto g(ax)$ again belongs to $D$. This proves continuity of fixed dilations. The identity coordinate makes the topology finer than usual and Hausdorff. Every basic neighborhood involves finitely many functions $g_j$, and

\[
g_j(\frac{1}{n})=\frac{g_j(1)}{n}\longrightarrow0.
\]
Hence $\frac{1}{n}\to0$ in $\tau_D$, proving nondiscreteness.

Every $\tau_f$ is contained in $\tau_D$. Conversely, each nonzero $g\in D$ is itself a Hamel coefficient functional and is continuous in $\tau_g$. Therefore the supremum contains every defining map of $\tau_D$, giving equality.

Fix $f\ne0$. Only $\mathfrak{c}$ members of $D$ are $\tau_f$-continuous, by Theorem~\ref{thm:rational-dual}, while $|D|=2^{\mathfrak{c}}$. Choose $g\in D$ outside that orbit. It is continuous for $\tau_D$ but not for $\tau_f$, so $\tau_f\subsetneq\tau_D$.

\end{proof}
More generally, for any family $\mathcal F\subseteq D$, the topology generated by the identity and all maps $x\mapsto f(ax)$, with $f\in\mathcal F$ and $a\in\mathbb R$, is a nondiscrete Hausdorff TVG topology. Thus $\tau_D$ is greatest in this explicitly defined class of topologies. This assertion does not identify a greatest nondiscrete TVG topology in the unrestricted class.

\section{Maximal extensions}
A topology is maximal nondiscrete if every strictly finer TVG topology is discrete. A greatest nondiscrete topology would contain every nondiscrete TVG topology. The latter is a stronger requirement.

\begin{proposition}[Existence of maximal extensions]\label{prop:maximal-extensions}
Every nondiscrete TVG topology on $\mathbb R$ has a maximal nondiscrete TVG extension. In particular, $\tau_D$ has such an extension, which is Hausdorff and strictly finer than every $\tau_f$.
\end{proposition}
\begin{proof}
Fix a nondiscrete TVG topology $\rho$. Order all nondiscrete TVG topologies containing $\rho$ by inclusion. For a nonempty chain, take the topology generated by its union. This is a TVG topology: group operations and fixed dilations are continuous for each generating topology and hence for the generated topology.

If the generated topology were discrete, finitely many zero-neighborhoods taken from members of the chain would have intersection $\{0\}$. One member of the chain contains all these finitely many open sets, and would therefore already be discrete. This contradiction proves nondiscreteness of the upper bound. The empty chain has $\rho$ as an upper bound. Zorn's lemma gives a maximal element. A strictly finer nondiscrete TVG topology would remain in the same ordered set, so the element is maximal in the asserted sense.

For $\rho=\tau_D$, every extension is Hausdorff because it contains the usual topology, and strictness above each $\tau_f$ follows from Theorem~\ref{thm:common-refinement}.

\end{proof}
The argument proves existence of maximal extensions, not uniqueness or existence of a greatest nondiscrete TVG topology. The unrestricted greatest-topology question is not resolved in this article.

\section{Literature and scope}
Our earlier papers \cite{Yaroshevskaya1975,Yaroshevskaya1981} provide the real-line topology $t$ and its minimality theorem; the English translation of the 1975 paper is also listed in \cite{Yaroshevskaya1975}. The intermediate constructions of \cite{Yaroshevskaya2026} lie strictly below the usual topology. The Hamel coefficient topologies considered here lie strictly above it, so the present construction complements that earlier direction of investigation.

Hamel's construction of discontinuous additive functions is classical \cite{Hamel1905}. Dom\'inguez and Tarieladze \cite[Example~4.3]{DominguezTarieladze2008} explicitly use such functions on the discrete real line in their discussion of character lifting. Dergachev \cite{Dergachev1978} discusses the TVG framework and a real-line example based on almost-periodic functions. Neither cited example is the nondiscrete coefficient topology defined in Section~\ref{sec:construction}.

Hejcman's work \cite{Hejcman1988} concerns TVG topologies on the real line. Our purpose here is to give a self-contained construction and an exact comparison theorem for rational-valued coefficient functionals; we make no historical priority claim relative to that work.

The classification in Theorem~\ref{thm:comparison} uses rational-valued coefficient maps. For arbitrary real-valued additive maps, Lemma~\ref{lem:finite-coordinates} still applies, but its reduction to a single rescaling need not follow. The multiplicity theorem counts topologies on the fixed real line, and does not count homeomorphism types.

\section*{Acknowledgement}
The author acknowledges the assistance of OpenAI's Codex in mathematical discussion, proof checking, exposition, and document preparation. The author assumes responsibility for the final content and any remaining errors.

\end{document}